\documentclass[11pt,a4paper, envcountsame,mathserif]{amsart}
\usepackage[usenames,dvipsnames]{color}

\usepackage[utf8]{inputenc}

\usepackage[capitalise]{cleveref}		

 \usepackage{amsmath, amssymb, xspace}
 \usepackage{graphicx}

 \usepackage[all]{xy}

\usepackage{tikz}
\usetikzlibrary{arrows,snakes,positioning,backgrounds,shadows}
\usepackage{stmaryrd}

\newtheorem{theorem}{Theorem}[section]

\newtheorem{thm}[theorem]{Theorem}

\newtheorem{fact}[theorem]{Fact}

\newtheorem{example}[theorem]{Example}

\newtheorem{proposition}[theorem]{Proposition}

\newtheorem{prop}[theorem]{Proposition}

\newtheorem{claim}[theorem]{Claim}

\newtheorem{conjecture}[theorem]{Conjecture}

\newtheorem{lemma}[theorem]{Lemma}

\newtheorem{cor}[theorem]{Corollary}

\newtheorem{question}[theorem]{Question}

\theoremstyle{definition}
\newtheorem{definition}[theorem]{Definition}

\newtheorem{remark}[theorem]{Remark}

\newcommand{\NN}{{\mathbb{N}}}

\newcommand{\QQ}{{\mathbb{Q}}}
\newcommand{\ZZ}{{\mathbb{Z}}}
\newcommand{\sub}{\subseteq}
\newcommand{\sN}[1]{_{#1\in \NN}}
\newcommand{\uhr}[1]{\! \upharpoonright_{#1}}
\newcommand{\ML}{Martin-L{\"o}f}
\newcommand{\SI}[1]{\Sigma^0_{#1}}
\newcommand{\PI}[1]{\Pi^0_{#1}}

\newcommand{\bi}{\begin{itemize}}
\newcommand{\ei}{\end{itemize}}
\newcommand{\bc}{\begin{center}}
\newcommand{\ec}{\end{center}}

\newcommand{\Halt}{{\ES'}}
\newcommand{\ES}{\emptyset}

\newcommand{\tp}[1]{2^{#1}}
\newcommand{\ex}{\exists}
\newcommand{\fa}{\forall}
\newcommand{\lep}{\le^+}
\newcommand{\gep}{\ge^+}
\newcommand{\la}{\langle}
\newcommand{\ra}{\rangle}

\newcommand{\seqcantor}{2^{ \NN}}

\newcommand{\cantor}{\seqcantor}
\newcommand{\strcantor}{2^{ < \omega}}

\newcommand{\Opcl}[1]{[#1]^\prec}

\newcommand{\n}{\noindent}

\newcommand{\vsp}{\vspace{6pt}}
\newcommand{\leb}{\mathbf{\lambda}}

\newcommand{\sss}{\sigma}

\renewcommand{\S}{S_\infty}

\newcommand{\lland}{\, \land \, }

\newcommand \seq[1]{{\left\langle{#1}\right\rangle}}

\newcommand\+[1]{\mathcal{#1}}

\newcommand{\wt}{\widetilde}
\newcommand{\ol}{\overline}

\newcommand{\lra}{\leftrightarrow}
\newcommand{\LR}{\Leftrightarrow}
\newcommand{\RA}{\Rightarrow}
\newcommand{\LA}{\Leftarrow}

\newcommand{\rapf}{\n $\RA:$\ }
\newcommand{\lapf}{\n $\LA:$\ }

\newcommand{\sssl}{\ensuremath{|\sigma|}}

\DeclareMathOperator{\Aut}{Aut}

\numberwithin{equation}{section}
\renewcommand{\hat}{\widehat}
\begin{document}

\title{Logic Blog 2018}

 \author{Editor: Andr\'e Nies}

\maketitle


 {
The Logic   Blog is a shared platform for
\bi \item rapidly announcing  results and questions related to logic
\item putting up results and their proofs for further research
\item parking results for later use
\item getting feedback before submission to  a journal
\item foster collaboration.   \ei

 \vsp
\begin{tabbing}

 \href{http://arxiv.org/abs/1804.05331}{Logic Blog 2017} \ \ \ \   \= (Link: \texttt{http://arxiv.org/abs/1804.05331})  \\
 
 \href{http://arxiv.org/abs/1703.01573}{Logic Blog 2016} \ \ \ \   \= (Link: \texttt{http://arxiv.org/abs/1703.01573})  \\
 
  \href{http://arxiv.org/abs/1602.04432}{Logic Blog 2015} \ \ \ \   \= (Link: \texttt{http://arxiv.org/abs/1602.04432})  \\
  
  \href{http://arxiv.org/abs/1504.08163}{Logic Blog 2014} \ \ \ \   \= (Link: \texttt{http://arxiv.org/abs/1504.08163})  \\

   \href{http://arxiv.org/abs/1403.5719}{Logic Blog 2013} \ \ \ \   \= (Link: \texttt{http://arxiv.org/abs/1403.5719})  \\

    \href{http://arxiv.org/abs/1302.3686}{Logic Blog 2012}  \> (Link: \texttt{http://arxiv.org/abs/1302.3686})   \\

 \href{http://arxiv.org/abs/1403.5721}{Logic Blog 2011}   \> (Link: \texttt{http://arxiv.org/abs/1403.5721})   \\

 \href{http://dx.doi.org/2292/9821}{Logic Blog 2010}   \> (Link: \texttt{http://dx.doi.org/2292/9821})  
     \end{tabbing}

\vsp

\n {\bf How does the Logic Blog work?}

\vsp

\n {\bf Writing and editing.}  The source files are in a shared dropbox.
 Ask Andr\'e (\email{andre@cs.auckland.ac.nz})  in order    to gain access.

\vsp

\n {\bf Citing.}  Postings can be cited.  An example of a citation is:

\vsp

\n  H.\ Towsner, \emph{Computability of Ergodic Convergence}. In  Andr\'e Nies (editor),  Logic Blog, 2012, Part 1, Section 1, available at
\url{http://arxiv.org/abs/1302.3686}.}

%

%

 
The logic blog,  once it is on  arXiv,  produces citations on Google Scholar.
%
\newpage
\tableofcontents

\part{Computability theory}
   
      \section{Nies and Stephan: randomness and $K$-triviality for measures} 
 
 \subsection{A randomness notion for measures} We consider algorithmically defined randomness notions    for  finite  measures on Cantor space $\cantor$ (usually probability measures). We use the letters $\mu, \nu$ etc for finite   measures, with $\leb$ reserved to the uniform measure.  Letters $\sss, \tau$ denote binary strings, $Z, X, \ldots$ elements of $\cantor$, $[\sss] = \{ Z \colon \,  Z \succ \sss\}$. So $\leb [\sss] = \tp{-\sssl}$.

 This research interacts with a recent attempt to define ML-randomness   for quantum states corresponding to infinitely many qubits \cite{Nies.Scholz:18}. (Probability  measures correspond to  the   quantum states $\rho$ where the matrix  $\rho \uhr {M_n}$ is diagonal for each $n$.)
 Here is the main definition,  which was   discussed during a meeting on effective dynamical systems in Toulouse March 2018, but is implicit in the earlier preprint \cite{Nies.Scholz:18}.  We now have a paper    on this \cite{Nies.Stephan:19}.
 \begin{definition} A measure $\mu$ is called \emph{\ML\ absolutely continuous} (ML-a.c.\ for short) if $\inf_m \mu(G_m) = 0$ for each ML-test $\seq{G_m}$. \end{definition}

  It   suffices to only consider descending ML-tests, because we can replace $\seq{G_m}$ by the ML-test $\hat G_m = \bigcup_{k >m} G_k$, and of course $\inf_m \mu(\hat G_m) = 0$ implies $\inf_m \mu(G_m) = 0$. So we can   change the passing  condition  to $\lim_m G_m = 0$.
 
 Also,   just as for bit sequences, it suffices to only  consider    the usual universal ML-test $U_m = \bigcup_{e< m} G^e_{m+e+1}$. So \ML\ a.c.\ ness  is  a $\PI 3$ property of measures. 
 
 Since $\bigcap_m U_m$ is the set $\+ C$ of non-MLR bit sequences, we obtain
 
 \bc $\mu$ is \ML\ a.c.\  $\lra \ \mu(\+ C) = 0$. \ec
 It follows that we can actually restrict the definition  to  any descending universal ML-test, such as  $\seq {\+ R_b}\sN b$ in the notation of  \cite[Ch.\ 3]{Nies:book}.

 Recall that a \emph{Solovay test } is a sequence $\seq {S_n}$ of uniformly $\SI 1$ sets such that $\sum_k S_k < \infty$. A bit sequence $Z$ passes such a test if $\forall^\infty k \, Z \not \in S_k$. We say that a measure $\mu $ \emph{passes} such a test if $\lim_k \mu(S_k) =0$.   For $Z\in \cantor$, we let $\delta_Z$ denote the  probability  measure that is concentrated on  $\{Z\}$. 
 \begin{fact} \label{fact:random_examples} \bi  \item[(i)] The uniform measure  $\leb$ is \ML\ a.c.\  \item[(ii)] $\delta_Z$ is \ML\ a.c.\  iff $Z$ is ML-random. 
 
 \item[(iii)] Let $\mu = \sum c_k \delta_{Z_k}$, for a sequence $\seq{c_k}$ of reals in $[0,1]$ with $\sum_k c_k = 1$. 
 
  $\mu$ is \ML\ a.c.\  iff  all the sets $ {Z_k}$ are ML-random. 
  
 \ei  \end{fact}
 \begin{proof} (i) and (ii) are immediate.

  \n  (iii)  $\RA$: If $Z_k \in \bigcap G_m$ for a ML-test $\seq{G_m}$  then $\mu(\bigcap G_m )  \ge \delta_k$, do $\mu$ is not \ML\ a.c.\ . 
   
   $\LA$:  given a ML-test $\seq{G_m}$, note that   the $Z_k$ pass this test as a Solovay test. Hence for each $r$, there is $M$ such that $Z_k \not \in G_m$ for each $k \le r$ and $m \ge M$. This implies that  $\mu(G_m) \le \sum_{k>r} c_k$. 
 \end{proof}
 
  The well known fact that ML-tests are equivalent to Solovay tests generalises to measures. 
 We use    the following  variant  for measures  of a result by Tejas Bhojraj  that he proved in the quantum setting.
 \begin{fact} A measure $\mu$  is \ML\ a.c.\  iff $\mu$ passes each Solovay test. \end{fact}
 \begin{proof} Each ML-test is a Solovay test. So the implication from  right to left is immediate. For the   implication from  left to right, suppose that $\seq {S_k} $ is a Solovay test and $\inf_k \mu(S_k) > \delta >0$. 
 We define a ML-test  that $\mu$ fails at level $\delta/2$. Let $S_{k,t}$  denotes the clopen set given by strings in $S_k$ of length $t$. By a minor modification of the standard proof (e.g.\ \cite[Prop.\ 3.2.19]{Nies:book}), let $G_{m,t} $ be the open set generated by strings $\sss$ such that       
  \bc $[\sss] \sub S_{k,t}$ for $\delta \tp{m-1}$ many $k$.  \ec  
  As in the standard proof one shows that $\leb G_{m,t} \le    \tp{-m}/\delta$. Then $G_m = \bigcup_t G_{m,t}$   (after thinning)   can be  turned it into a ML-test. 
 
 Given $m$ we pick  $t\in \NN$ sufficiently large so that for some  set  $M\sub \{0, \ldots, t-1\}$ of size $\tp m$  we have $\mu (S_{k,t}) > \delta$ for each $k \in M$.  We show that $\mu(G_{m,t}) > \delta/2$. Let $\sss$ range over strings of length $t$.  We have   \bc  $\sum_{[\sss] \not \subseteq G_{m,t}} \sum_{k \in M}\mu[\sss] \le 2^{m-1} \delta$ \ec  by definition of $G_m$. Since  $\tp m \delta \le \sum_{k \in M} \mu(S_{k,t})$, this implies     
 \bc $\sum_{[\sss]  \subseteq G_m} \sum_{k \in M}\mu[\sss] > 2^{m-1} \delta$. \ec   Since $|M| = 2^m$ this shows $\mu G_{m,t} > \delta/2$ as required.
 \end{proof} 
  
 For a measure $\nu$ and string $\sss$ with $\nu[\sss] >0$ let $\nu_\sss$ be the localisation: $\nu_\sss(A) = \tp{-\sssl} \nu (A \cap [\sss])$. Clearly if $\nu$ is ML-a.c.\  then so is $\nu_\sss$.

 A set $S$ of probability measures is  called \emph{convex} if $\mu_i \in S$ for $i \le k$ implies  that the convex combination $\mu = \sum_i \alpha_i \mu_i \in S$, where the $\alpha_i $ are reals in $ [0,1]$ summing up to $1$.   The extreme points of $S$ are the ones that can only be written as  convex combinations of length 1 of   elements of $S$. 
 \begin{prop} The \ML\ a.c.\  probability measures form a convex set. Its extreme points are the  Dirac measures.
 \end{prop} 
 \begin{proof} For convexity, suppose $\seq {G_m}$ is a descending ML-test. Then
 
 \bc  $\lim_m \mu_i(G_m) > 0$ for each $i$, \ec and hence $\lim_m \mu(G_m) >0$. 
 
 If $\mu$ is a Dirac measure then  it is an extreme point. Conversely,  if $\mu$ is not Dirac there is a least number $t$ such that the decomposition \bc  $\mu = \sum_{\sssl = t, \mu[\sss] >0} \mu [\sss] \cdot \mu_\sss$  \ec is nontrivial. 
 Hence $\mu$ is not an extreme point. 
 \end{proof}

 \subsection{Initial segment complexity of a measure $\mu$  as  a $\mu$-average}
 Let $K(\mu \uhr n)= \sum_{|x| =n}K(x) \mu[x] $ be the $\mu$-average of all the $K(x)$ over all  strings~$x$   of length $n$.    In a similar way  we define $C(\mu \uhr n)$. 
 
 \begin{fact} $C (\leb \uhr n) \gep n$, and therefore $K (\leb \uhr n) \gep n$ \end{fact} 
 \begin{proof} Suppose $d$ is chosen so that for each $x$ we have $C(x) \le |x| + d$ (we can in fact  ensure $d=1$ with the right universal machine, see \cite[Ch 2]{Nies:book}). 
  
  \begin{eqnarray*} C (\leb \uhr n) & = & \sum_{r=0}^{n+d} \sum_{x: |x| =n \land C(x) \ge r} \tp{-n} \\ 
    & \ge  & \sum_{r = 0} ^{ n } [  \sum_{|x| = n } \tp{-n}  - \sum_{|x| = n , C(x) < r)} \tp{-n}] \\
         &  \ge &  n+1 - \sum_{r \le n} \tp{-n+r}  \ge n-1. \end{eqnarray*} This does it. \end{proof}
 
We say that $\mu$ has \emph{complex  initial segments}  if $K(\mu \uhr n) \gep n$. The analog of Levin-Schnorr fails for measures in both directions.
\begin{example} There is a \ML\ a.c.\  measure  $\mu$ such that 

\n  $\sup  (n - K(\mu \uhr n))   = \infty$. 
\end{example}
\begin{proof} We let $\mu = \sum c_k \delta_{Z_k}$ where $Z_k$ is ML-random and $0^{n_k} \prec Z_k$ for a sequence $\seq {c_k}$ of reals in $[0,1]$ that add up to $1$, and a sufficiently fast growing sequence $n_k$. Such a $\mu$ is \ML\ a.c.\  by Fact~\ref{fact:random_examples}. 

For $n_k \le n < n_{k+1}$ we have 

$K(\mu \uhr n) \le (\sum_{l=0}^k c_l) \cdot  (n +2 \log n)  + \sum_{l=k+1}^\infty c_l \cdot 2 \log n) $

 \hfill $  \le (1-c_{k+1}) n + 2 \log n$. 

So if we ensure that $c_{k+1} \cdot n_k > k + 2 \log n_k$  we are good.   For instance, we can  let $c_k = 1/(k (k+1))$ and $n_k = \tp{k+4}$. 
\end{proof} 

We falsify the converse implication by the following. 
\begin{theorem} There are a random $X$ and a non-random $Y$ such that,
for all $n$, $K(X \upharpoonright n)+K(Y \upharpoonright n) \geq^+ 2n$. \end{theorem}

\begin{proof}
Let $X$ be a low Martin-L\"of random set. There is strictly growing
function $f$ such that the complement of the image of $f$ is a recursively
enumerable set $E$ and $K(X \upharpoonright m) \geq m+3n$
for all $m \geq f(n)$. Note that this function exists, as $X$ is
low and Martin-L\"of random and so, for all $n$, the
maximal $m$ such that $K(X \upharpoonright m) \leq m+3n$ can be
found in the limit.

Now let $g(n) = \max\{m: f(m) \leq n\}$.
By  a result of Miller and Yu~\cite[Cor.\ 3.2]{Miller.Yu:11}, there is a Martin-L\"of
random $Z$ such that there exist infinitely many $n$ with
$K(Z \upharpoonright n) \leq n+g(n)/2$. For this set $Z$,
let  \bc $Y = \{n+f(n): n \in Z\}$. \ec Note that
$K(Z \upharpoonright n) \leq K(Y \upharpoonright n)+g(n)+K(g(n))$,
as one can enumerate the set $E$ until there are, up to $n$, only
$g(n)$ many places not enumerated and then one can reconstruct
$Z \upharpoonright n$ from $Y \upharpoonright n$ and
$g(n)$ and the last $g(n)$ bits of $Z$. As $Z$ is Martin-L\"of random,
$K(Z \upharpoonright n) \geq^+ n$ and so,
\bc $K(Y \upharpoonright n) \geq^+ n-g(n)-K(g(n)) \geq^+ n-2g(n)$. \ec
The definitions of $X,f,g$ give $K(X \upharpoonright n) \geq n+3g(n)$.
This shows that $K(X \upharpoonright n)+K(Y \upharpoonright n) \geq 2n$
for almost all $n$.

However, the set $Y$ is not Martin-L\"of random, as there are infinitely
many $n$ with $K(Z \upharpoonright n) \leq^+ n+g(n)/2$. Now
$Y \upharpoonright n+g(n)$ can be computed from $Z \upharpoonright n$
and $g(n)$, as one needs only to enumerate $E$ until the $g(n)$ nonelements
of $E$ below $n$ are found and they allow to see where the zeroes have to
be inserted into the string $Z \upharpoonright n$ in order to obtain
$Y \upharpoonright n+g(n)$. Note furthermore, that $K(g(n)) \leq g(n)/4$
for almost all $n$ and thus $K(Y \upharpoonright n+g(n)) \leq^+ n+3/4 \cdot
g(n)$ for infinitely many $n$, so $Y$ cannot be Martin-L\"of random.
\end{proof}

Note that the measure $\mu = (\delta_X+\delta_Y)/2$ has only two equal-weighted
atoms and furthermore satisfies that one of these atoms is not
Martin-L\"of random. So every component of a universal Martin-L\"of test
has at least $\mu$-measure $1/2$. On the other hand, $K(\mu \upharpoonright n)
\geq n$ for almost all $n$ by the preceding result. Thus one has the following
corollary.

\begin{cor} There is a measure $\mu$ with complex initial segments which
is not \ML\ a.c.\   \end{cor}

\begin{prop} Suppose that $\mu $  is a  measure  such that  $K(\mu \uhr n)  \ge n+K(n)-r$ for infinitely many $n$. Then $\mu $ is \ML\ a.c.\  \end{prop}

\begin{proof} Suppose that $\mu $ is not \ML\ a.c.\  So  there is a ML-test $\seq {G_d} \sN d$ and $\epsilon > 0$  such that  $\mu (G_d) > \epsilon$ for each $d$.   If $x$ is a string of length $n$ such that $[x] \sub G_d$ then 
\[ K(x \mid n,d ) \lep n - d.  \]
To see this let $M$ be the machine that on  a pair of auxiliary inputs $n,d  $  gives a description of length $n-d$ for each such $x$ (so the descriptions for different $x$ are prefix free). It follows that  for $x$ as above
\[ K(x)  \lep n + K(n) - d + 2 \log d.  \]
Now view $G_d$ as given by an enumeration of strings, and choose $n$ large enough so that $\mu G_d^{\le n} > \epsilon$, where $G_d^{\le n}$ denotes the open set given by the strings in this enumeration of length at most $n$. Let $c$ be a constant such that $K(x) \le n + K(n) + c$ for each $x$ of length $n$. We have 
\begin{eqnarray*}K(\mu \uhr n)   & = &  \sum_{|x| =n}K(x) \mu[x]  \\
      & = &   \sum_{|x| =n, [x] \sub G_d } K(x) \mu[x]   +  \sum_{|x| =n, [x] \not \sub G_d } K(x) \mu[x] \\
      & \le &    n+ K(n) + c - \epsilon d/2. \end{eqnarray*}
The last inequality holds because  \bc $  \sum_{|x| =n, [x] \sub G_d } K(x) \mu[x]   \le \epsilon (n+ K(n) + c - d  + 2\log d)$ \ec and 
\bc $  \sum_{|x| =n, [x] \not \sub G_d } K(x) \mu[x] \le (1- \epsilon) (n + K(n) + c)$. \ec 

Now given $r$ let $d = 2 r / \epsilon$. By the above, for large enough $n$ we have $K(\mu \uhr n)  <  n+K(n)-r$. So $\mu $ is not strongly Chaitin random.  \end{proof}

\begin{question}  In analogy to the case of bit strings, does strong Chaitin randomness of a measure imply \ML\ a.c.\ ness relative to $\Halt$? \end{question}

If the measure $\mu$ has an atom $A$ but is not Dirac then
function $C(\mu \uhr n)$ is not bounded from below by $n-c$
for any $c$. The reason is that when $c' = \mu(A)$ then
for this atom, the function $n \mapsto c' \cdot (n-C(A \upharpoonright n))$
is not bounded by any constant and therefore it can go arbitrarily
low; this would then make the average to be below $n-c$ for any given
$c$ at infinitely many $n$.
 
 A \emph{fan}  is a prefix-free set $V$ such that $\Opcl V = \cantor$. Note that $V$ is necessarily finite. The $\mu$-\emph{average length} of $V$ is \bc $\ell(\mu\uhr V ) = \sum_{\sss \in V} \sssl \mu[\sss]$. \ec Generalising the above, we let 
 
\bc  $K (\mu \uhr V)= \sum_{\sss \in V} K(\sss) \mu [\sss]$.  \ec

We say that $\mu $ has  complex initial segments  in the  strong sense  if

\n  $K (\mu\uhr V) \gep \ell(\mu\uhr V)$ for each fan $V$. To be done: does this imply \ML\ a.c.\ ?

  \subsection{Connection to ML-randomness of measures in $\+ M(\cantor)$}  A natural probability measure $\mathbb P$ on the  space $\+ M(\cantor)$  of probability measures on Cantor space has been introduced implicitly  in Mauldin and Monticino \cite{Mauldin.Monticino:95},  and in Quinn Culver's thesis \cite{Culver:15} in the context of computability, where he shows that this measure is computable.  Let $\+ R \sub [0,1]^{\strcantor}$ be the closed set of representations of probability measures; namely, $\+ R$ consists of those $X$ such that $X_\sss = X_{\sss0} + X_{\sss 1}$ for each string $\sss$. $P$  is the unique measure on $\+ R$ such that for each string $\sss$ and $r, s \in [0,1]$, we have 
\bc $P(X_{\sss0} \le r \mid X_\sss = s) = \min (1, r/s)$. \ec 
That is, we choose $X_{\sss 0}$  at random w.r.t. the uniformly distribution in the interval $[0, X_\sss]$), and the choices made at different strings are independent.

%
 
  \begin{prop} \label{fffuuukkk}Every probability measure $\mu $ that is ML-random wrt to $\mathbb P$  is \ML\ a.c.\ . \end{prop}
  The proof is based on    two facts.   For  $G \sub \cantor $ be open, for the duration of this proof  let $\mu$ range over ${\+ M (\cantor)}$  and let  \bc $r_G= \int \mu(G) d \mathbb P(\mu)$. \ec
  \begin{fact}  $r_G = \leb (G)$. \end{fact}

\begin{proof} Clearly  for each $n$ we have \[\sum_{\sssl = n}r_{[\sss]} = \int \sum_{\sssl = n} \mu ([\sss]) d\mathbb P(\mu) =1.\] Further,  $r_\sss = r_\eta$ for $\sssl = |\eta|=n$ because there is a $\mathbb P$-preserving transformation $T$ of $\+ M(\cantor)$ such that $\mu([\sss]) = T(\mu)([\eta])$. Therefore $r_{[\sss]} = \tp{-\sssl}$. 

If $\sss, \eta $ are incompatible then $r_{[\sss] \cup [\eta]} = r_{[\sss]} + r_{[\eta]}$. Now it suffices to write $G = \bigcup_i [\sss_i]$ where the strings $\sss_i$ are incompatible, so that $\leb G = \sum_i \tp{-|\sss_i|}$. 
\end{proof}

\begin{fact} Let $\mu \in \+ M(\cantor)$ and let  $\seq{G_m}$ be a   ML-test  such that there is $\delta \in \QQ^+$ with $\fa m \,  \mu(G_m) > \delta$. Then $\mu $ is not ML-random w.r.t.\ $\mathbb P$. 
\end{fact}
\begin{proof} Observe that by the foregoing fact   \[\delta \cdot  \mathbb P(\{ \mu \colon \mu(G_m) \ge  \delta\})  \le \int \mu(G_m) d \mathbb P(\mu)  = \leb(G_m) \le \tp{-m}.\] 
Let  $\mathcal G_m = \{ \mu \colon \mu(G_m) >  \delta\}$ which is uniformly $\SI 1$ in $\+ M(\cantor)$.  Fix $k$ such that $\tp{-k} \le \delta$; then $\seq {\mathcal G_{m+k}} \sN m$ is a ML-test w.r.t.\ $\mathbb P$ that succeeds on $\mu$. 
\end{proof}

 Culver shows that each ML-random $X$ for $P$ is non-atomic. So by Fact~\ref{fact:random_examples} the converse of Prop.\ \ref{fffuuukkk} fails: not every \ML\ a.c.\  $X$ is ML-random with respect to $P$.

 \subsection{SMB theorem} We recall some notation from the 2017 Logic Blog, Section 6.2, adapting some letter uses.  
    $\mathbb A^\infty$ denotes the space of one-sided infinite sequences of symbols in $A$. We can assume that this is the sample space, so that $X_n(\omega) = \omega(n)$.  By $\mu $ we denote their joint distribution.  A dynamics on $\mathbb A^\infty$  is given by the shift operator $T$, which erases the first symbol of   a sequence.
  A measure $\mu$ on $A^\infty$ is \emph{$T$-invariant} if $\mu G = \mu T^{-1}(G)$ for each measurable $G$. 
 
  We consider the r.v.  \[ h^\mu_n(Z ) = -\frac 1 n \log \mu [Z\uhr n],\]  (recall that  $\log$ is w.r.t.\ base $2$). 
 
Recall that $\mu$ is \emph{ergodic} if   every  $\mu$ integrable function $f $ with $f \circ T = f$  is constant $\mu$-a.s.  
An equivalent  condition that is easier to check is the following: for  $u,v \in \mathbb A^*$,  
\[ \lim_N \frac 1 N \sum_{k=0}^{n-1} \mu ([u] \cap T^{-k}[v]) = \mu [u] \mu [v]. \]

For ergodic $\mu$, the entropy $H(\mu)$ is defined as $\lim_n H_n(\mu)$, where \[H_n(\mu) =  -\frac 1 n \sum_{|w| = n}  \mu [w] \log \mu [w].\]
One notes that $H_{n+1}(\mu) \le H_n(\mu)\le 1$ so that the limit exists. Also note that $H_n(\mu) = \mathbb E h^\mu_n$.

The following says that in the ergodic case, $\mu$-a.s.\ the empirical  entropy equals the entropy of the measure.

\begin{thm}[SMB   theorem] Let  $\mu $ be an  ergodic    invariant measure for the shift operator $T$ on the space $\mathbb A^\infty$. 
Then for $\mu$-a.e.\  $Z$   we have  $\lim_n h^\mu_n(Z) = H(\mu)$.
\end{thm}

If $\mu$ is computable, then the conclusion holds for $\mu$-ML-random $Z$ by results of Hochman (implicit) \cite{Hochman:09} and Hoyrup \cite{Hoyrup:12}. Recent work of A.\ Day extends this to   spaces  other than $\mathbb A^\infty$ and amenable group actions. Here we keep the space but change the type of object. We say that a measure $\rho$ is $\mu$-\ML\ a.c.\  if $\rho(G_m) \to 0$ for each $\mu$-ML test $\seq{G_m}$.  Here is a special case of Conjecture 6.3 in 2017 Logic Blog where the states $\rho, \mu$ when restricted to   the matrix algebra $M_n$ are diagonal. 
Enough patience will suffice. 
\begin{conjecture}[Effective SMB   theorem for measures] Let  $\mu $ be a computable   ergodic      invariant measure for the shift operator $T$ on the space $\mathbb  A^\infty$. Suppose the measure $\rho$ is $\mu$-\ML\ a.c.\ .
Then  $\lim_n E_\rho h^\mu_n = H(\mu)$. 
\end{conjecture} 
 
 \subsection{$K$-triviality for measures}

 \begin{definition} A measure $\mu$ is called $K$-trivial if $K(\mu \uhr n) \lep K(n)$. \end{definition}
 For Dirac measures $\delta_A$ this is the  same as saying that $A$ is $K$-trivial in the usual sense. 
 
 \begin{prop} Suppose $\mu$ is $K$-trivial. Then $\mu$ has atoms. In fact, $\mu$ is concentrated on its atoms. \end{prop}
 \begin{proof}  For each $c$ there is $d$ (in fact $d = O(2^c)$) such that for each $n$ there are at most $d$ strings $x$ of length $n$ with $K(x) \le K(n) + c$. Since $\mu$ is non-atomic, there is $n$ so that for each $x$ of length $n$ we have $\mu[x] \le 1/2d$. Note that there is a constant $b$ such that  $K(x) \ge K(|x|)-b $ for each $x$. Then  we have  that in the $\mu$- average
  $K(\mu \uhr n) $, the $x$ of length $n$ such that $K(x) \le K(n) + c$ have total measure  at most $1/2$,  and each $K(x) \ge K(n) -b$. So the $\mu$ average  is at least 
 $K(n) + c/2 $ up to a constant.  
 \end{proof} 
 
 \begin{prop} For each order function $f$ there is a non-atomic measure $\mu $ such that $K(\mu \uhr n) \lep K(n) + f(K(n))$. 
 
 In fact, for each nondecreasing unbounded  function $f$ which is approximable from
above there is a non-atomic measure $\mu$ such that
$K(\mu \upharpoonright n) \leq^+ K(n)+f(n)$.
\end{prop}

\begin{proof}   
 There is a recursively enumerable
set $A$ such that, for all $n$, $A$ has up to $n$ and up to a constant $f(n)/2$
non-elements. One let $\mu$ be the measure such that
$\mu(x) = 2^{-m}$ in the case that all ones in $x$ are not in $A$
and $\mu(x) = 0$ otherwise, here $m$ is the number of non-elements
of $A$ below $|x|$. One can see that when $\mu(x) = 2^{-m}$
then $x$ can be computed from
$|x|$ and the string $b_0 b_1 \ldots b_{m-1}$ which describes the
bits at the non-elements of $A$. Thus $K(x) \leq^+
K(|x|)+K(b_0 b_1 \ldots b_{m-1}) \leq^+ K(|x|)+2m$.
It follows that $K(\mu \upharpoonright n) \leq^+ K(n)+f(n)$,
as the $\mu$-average of strings $x \in \{0,1\}^n$
with $K(x) \leq^+ K(n)+f(n)$ is at most $K(n)+f(n)$ plus a constant.
\end{proof}

REMARK. Note that when $f$ is a recursive order function or an
order function which is approximable from above then
there is a further order function $f'$ which is approximable
from above such that $f'(n) \leq^+ f(K(n))$ for all $n$;
one just chooses $f'(n) = \min\{f_s(K_s(m)): m \geq n, s \geq 1\}$.
Thus one can bring the above result into the form that for all
recursive order functions $f$ there is a measure $\mu$ satisfying
$K(\mu \upharpoonright n) \leq K(n)+f(K(n))$.

 The $K$-trivial measure form a convex set. However it is not closed under infinite convex sums. One takes finite sets which pointwise converge to
$\Omega$ and let the $c_k$ fall sufficiently slowly so that at level
$n$ there is still $(n+2)^{-1/2}$ measure on $\Omega \upharpoonright n$
and therefore the corresponding $\mu$-average grows like the squareroot of $n$,
and not like $K(n)$.

 In more detail, let $A_k = \{\ell: \ell \in \Omega \wedge \ell < k\}$ and
$c_k = (k+1)^{-1/2}-(k+2)^{-1/2}$. All sets $A_k$ are finite and thus $K$-trivial. Furthermore, the
sum of all $c_k$ is $1$.

 Let $\mu = \sum_k c_k \cdot \delta_{A_k}$.
Then 
$A_\mu(n) = \sum_{x \in \{0,1\}^n} \mu(x) \cdot K(x)
     \geq (\sum_{m \geq n} c_m) \cdot K(\Omega\upharpoonright n)
     \geq (n+2)^{-1/2} \cdot (n+2) = \sqrt{n+2}$ for almost all $n$
and thus the average grows faster than $K(n)+constant$.
So the measure is not $K$-trivial.

We call a measure $\mu$ \emph{low for $K$}  if for each $z$
\[   \int K^X(z) d\mu (X) \gep K(z). \]
Thus we form the $\mu$-average over all oracles. Clearly if $A$ is low for $K$ as a set then $\delta_A$ is low for $K$.   Merkle and Yu have shown that $\leb$ is low for $K$. So lowness for $K$ does not imply $K$-triviality.   
It would still be interesting to relate  lowness for $K$ with $K$-triviality in the case of measures.

      \section{Yu: A note on $\delta^1_2$}
 Let $\delta^1_2$ be the least ordinal that cannot be presented by a $\Delta^1_2$-well ordering over $\omega$ and $\delta^{1,x}_2$ be the one  relative to $x$. Define $$\Delta_{12}=\{x\mid \delta^{1,x}_2=\delta^1_2\}.$$
 
The following  result must be well known but I have not found a reference.
 \begin{proposition}
 $\Delta_{12}$ is $\Delta^1_3$ but neither $\Sigma^1_2$ nor $\Pi^1_2$.
 \end{proposition}
 \begin{proof}
 $x\in \Delta_{12}$ if and only if every $\Pi^1_1(x)$-singleton $z$ coding a well ordering of $\omega$ is bounded by a $\Pi^1_1$-singleton $z_0$ coding a  well ordering of $\omega$ if and only if there is a real $s$ coding a well ordering of $\omega$ such that $L_s$ contains all the $\Pi^1_1(x)$-singletons and every real in $L_s$ is $\Delta^1_2$. So $\Delta_{12}$ is a $\Delta^1_3$-set.
  
 Since every nonempty $\Sigma^1_2$-set contains a $\Delta^1_2$-member, we have that  $\Delta_{12}$ is not $\Pi^1_2$. Now suppose that  $\Delta_{12}$ is   $\Sigma^1_2$. Then let $y\in L$ be a real computing all the $\Delta^1_2$-reals. Then every $y$-random is $\Delta^1_2$-random. By the assumption, $R=\{r\mid r \mbox{ is }y\mbox{-random and } r\in \Delta_{12}\}$ is a $\Sigma^1_2(y)$-set. If $V$ contains an $L$-random real, then $R$ is not empty. So by Shoenfield's absoluteness, $R$ is not empty and contains a real $r\in L$. But every real $r\in R$ must be $L$-random, a contradiction. 
 \end{proof}

 

  \part{Group  theory and its connections with logic}
   \section{Nies and Schlicht: the normaliser of a finite permutation group} 
Let $G$ be a group.  The group of inner automorphisms $\mathrm{Inn}(G)$  forms a normal subgroup of $\mathrm{Aut}(G)$. The quotient group is called the group of outer automorphisms,   denoted by  $\mathrm{Out}(G)$. For instance, $\mathrm{Out}(S_6)$ has 2 elements. For more examples, note that since $G'$ is invariant,  there is a canonical surjection $\mathrm {Aut }(G) \to \mathrm{Aut}(G_{ab})$ with kernel containing $\mathrm{Inn}(G)$.  In the case of $F_2$ equality holds, so that $ \mathrm{Out}(F_2) \cong GL_2(\ZZ)$.

It is well-known that no cyclic group of odd order is of the form $\Aut(G)$ for any group $G$. On the other hand, every finite group is the outer automorphism group of some group $N$ which can be chosen to be fundamental  group of a closed hyperbolic 3-manifold (a result of Sayadoshi Kojima). See \cite{Bumagin.Wise:02} for background.

Here is a simple (and known) fact.  Given a finite group $G$ with domain $\{0,\ldots, n-1\}$, we think of $G$ as embedded into $S_n$ via the left regular representation $ g \to \tau_g$ where $\tau_g(x) = gx$.  (E.g.\ for $G= S_6$, we have $n=720$.)  Let $N_G$ denote the normaliser of $G$ in $S_n$. 

\begin{prop} \label{prop:Sn_normaliser} 

There is a canonical surjection $R \colon N_G \to \mathrm{Aut}(G)$ mapping $G$ to $\mathrm{Inn}(G)$, thereby showing that  $N_G/G $ is   isomorphic to $\mathrm{Out}(G)$. \end{prop}

\begin{proof} A canonical map $R\colon N_G \to \mathrm{Aut}(G)$ is defined  by 
\bc $R(\phi) (g) = h$ if $\phi \tau_g \phi^{-1} = \tau_h$. \ec
Clearly $R(\phi)$ is an automorphism of $G$ for each $\phi \in N_G$. 
To check that $R$ is a homomorphism, note that  for $\phi, \psi  \in N_G$ \bc $R(\phi \psi)(g) = h \LR \phi \psi \tau_g \psi^{-1} \phi^{-1}= \tau_h$. \ec Now   $\tau_{\psi(g)} = \psi \tau_g \psi^{-1}$. So the equation above implies  $R(\phi) (R(\psi)(g)) = h$. 

$R$  is a surjection:  $\phi \in \mathrm{Aut}(G)$ iff  $R(\phi) = \phi$ (where we identify $g$ with $\tau_g$).

Finally,  $\phi = \tau_u$ iff $\phi \tau_g \phi^{-1} = \tau_{u g u^{-1}}$ for each $g$ iff  $R(\phi)$ is inner. \end{proof}
 
 In fact we don't need that  $G$ is finite.

    \section{Kaplan, Nies,  Schlicht and others: closed subgroups of $\S$}
 
 We make some remarks on closed subgroups of $\S$. These are the automorphism groups of   structures with domain $\omega$. We in particular consider the following kind. 
 A closed subgroup $G$ of $\S$ is called \emph{oligomorphic} if it has only finitely many $n$-orbits, for each $n$. These are the automorphism groups of the $\aleph_0$-categorical structures with domain $\omega$.  We say that a topological  group $G$  is   \emph{quasi-oligomorphic} if it is in a topological group isomorphism with an oligomorphic group.  
 
 \subsection{The centre}

Let $G$ be a closed subgroup of $\S$.  For $p \in \omega$, by $k_p(G)$ we denote the number of  orbits  of the natural action of $G$  on $\omega^p$; such orbits will be   called $p$-orbits. (The parameter  $k_p(G)$ is denoted $F_p^*(G)$ in \cite{Cameron:90}.)  For $r \in \omega$ let $k_2(G,r)$ denote the number of $2$-orbits containing a  pair of the form $(r,t)$ (which only depends on the 1-orbit of $r$). 
 Suppose that $k_1(G) = n$ and let  $r_1, \ldots, r_n \in \omega$ represent the $1$-orbits.
 
 \begin{fact} $|C(G)| \le \prod_{i \le n} k_2(G, r_i)$. In particular, if $G$ is $1$-transitive then the size of the centre  is at most $k_2(G)$. \end{fact}
 
 \begin{proof} Write $C = C(G)$. 
 For any $r \in \omega$, and $c , d\in C $, if $cr \neq d r$ then $(r,cr)$ and $(r, dr)$ are in different $2$-orbits. Hence $|Cr| \le k_2(G,r)$.
 
 Consider now  the natural  left action $C \curvearrowright  \prod_i C   r_i$ (Cartesian product of sets). If $c , d\in C $ are distinct then   $c g r_i \neq d g r_i$ for some $i$ and $g \in G$, so that $c r_i \neq d r_i$. Hence $|C| \le \prod_{i \le n} |C \cdot r_i|$. This shows the required bound. \end{proof}

 Greg Cherlin suggested another way to prove that  for oligomorphic $G$ the centre $C(G)$  is finite   (but without an explicit bound on its size).  We may suppose $G = \Aut(M)$ for an  $\aleph_0$-categorical structure $M$  with domain~$\omega$. Note that $\Phi \in \Aut(M)$ is definable as a  binary relation on $M$ iff $\Phi$ is invariant under the natural action of  $G$ on $M$, i.e., $\Phi \in C(G)$.  Since there are only finitely many definable binary relations, $C(G)$ is finite.

  We consider   examples of oligomorphic groups $G$ with a nontrivial centre.  
  
 \n  1. For any finite abelian group $A$, the natural action of $A \times \S$ on $A \times \omega$ yields an 1-transitive oligomorphic group with centre $A$. The number of $2$-orbits is at least $|A|$ because $(an, a'n') \approx (bm, b'm')$ implies that $a-b = a'-b'$.  (In fact it is $2 |A|$.) 
 
 \n 2. Let $M_1$ be the structure with one equivalence relation $E$ that has all classes of size $2$; say,    for $x,y \in \omega$
\bc  $ x E y \lra  x \mod 2 = y \mod 2$.  \ec
Write $C_k = \ZZ/ k\ZZ$.  We have $ G_1 := \Aut(M_1) = C_2 \wr \S$, where $\wr$ denotes the unrestricted wreath product.   Here $\S$ is viewed with its natural action on  $\omega$; so $G $ is an extension of $L= C_2^\omega$ by $\S$ with $\S$ acting by $\phi \cdot f = f \circ \phi$, for $f\in L$.  (Note that $L$ is the automorphism group of the structure where the individual equivalence classes are now  distinct unary predicates.  $G$ is the normaliser of $L$ in $\S$.)   The centre $C(G)$ consists of the identity and the  automorphism that maps  each element to the other one in its $E$ class.  The centre of $G/C(G)$ is trivial. 
 
 \n 3. Let $M_2$ be the structure with   equivalence relations $E \sub F$ such that each $E$-class has size 2 and each $F$-class has size 4. Then  the automorphism group of a single $F$-class is $C_2 \wr C_2$, and hence $G_2 := \Aut (M_2)  = (C_2 \wr C_2) \wr \S)$.    As before,  the centre $C(G_2)$ consists of the identity and the  automorphism  that switches each element in its $E$ class.  We have $G_2 / C(G_2)  \cong G_1$. 
 
 Similarly,   for each $n$ there is an oligomorphic group $G_n = \Aut(M_n)$ with a  chain of $n$ higher centres.  
 
 \subsection{The central quotient} The main purpose in this section is to show that \emph{for oligomorphic $G$, the central quotient $G/C(G)$ is quasi-oligomorphic}. Some facts needed along the way hold in more generality.
 
 Suppose a group $G$ acts on a set $X$ and $N \trianglelefteq G$. Write $\sim_N$ for the orbit equivalence relation of the subaction of $N$. Note that $G$ acts naturally on $X/\sim_N$ via $g\cdot [x] = [g\cdot x]$ (where $[x]$ is the $\sim_N$ class of $x$).  Since elements of $N$ act as the identity, $G/N$ acts on $X / \sim_N$. 
 
 Suppose now $G$ initially acts faithfully on a set $Y$, say $Y = \NN$. Let $X =  Y \times Y$ and let  $G$ act on $X$ by the usual diagonal action. Let $N= C(G)$.  
 
 \begin{fact} The action of $G /C(G)$ on $X / \sim_{C(G)}$ is faithful. \end{fact}
 
 To see this, suppose $g \not \in N =C(G)$. So grab $\eta \in G$ such that \bc $ (g \eta) \cdot w \neq (\eta g) \cdot w$. \ec  Let $w'= \eta \cdot w$.  Then $ g \cdot (w,w') \not \sim_N (w,w')$,  because any element $h $ of $G$ such that $h \cdot (w,w') = g \cdot (w,w')$ satisfies $(h \eta) \cdot w \neq (\eta h) \cdot w$, so that $h \not \in C(G)$.

 We now switch to topological setting. Given  a Polish group $H$  with a faithful  action $\gamma \colon H \times V \to V$, say  for a  countable set $V$,  we obtain a monomorphism $\Theta_\gamma \colon H \to S_V$ given by $\Theta_\gamma(g)(k)= \gamma(g,k)$. A Polish group action is continuous iff it is separately continuous (i.e. when one argument is fixed). In the case of an action on countable $V$ (with the discrete topology), the latter condition  means that  
 
 \bi \item[(a)] for each $k,n \in V$, the set $\{ g \colon \gamma(g,k)=n\}$ is open. \ei

 So $\gamma$ is continuous iff $\Theta_\gamma$ is continuous. 

\begin{definition} 
We say that a faithful  action $\gamma \colon H \times V \to V$   is \emph{strongly continuous} if the embedding $\Theta_\gamma$ is topological. 
\end{definition} 

Equivalently, the action is continuous,  and for each neighbourhood $U$ of $1_H$, also $\Theta_\gamma(U)$ is open, namely,   

 \bi \item[(b)]
for each neighbourhood $U$ of $1_H$, there is finite set $B \sub V$  such that $\fa k \in B \, \gamma(g,k) = k$ implies $g \in U$.  \ei Since $H$ is Polish, strong continuity of the action  implies that $H$  is topologically isomorphic  via $\Theta_\gamma$ to a closed subgroup of $S_V$ (see e.g.\  \cite[Prop.\ 2.2.1]{Gao:09}).
 
 Now consider the case that $Y = \NN$ and  $G$ is a closed subgroup of $\S$. Since $C(G)$ is closed, $H= G/C(G)$ is naturally a Polish group via  the quotient topology:  for $C \le U \le G$, the   subgroup $U/C \le H$ is declared to be open iff $U$ is open in $G$.  (See e.g.\  \cite[Prop.\ 2.2.10]{Gao:09}.)

Let  $X = Y \times Y$ as above. Suppose that $V:= X/\sim_{C(G)}$ is infinite (e.g.\ when $C(G)$ is finite),  so through the action $\gamma$ above we obtain  an (algebraic) embedding $\Theta_\gamma$ of $G/C(G)$ into $S_V$ (which can be identified with $\S$).  
 
 \begin{claim}  \label{cl:1tr} Suppose that $G$ is a closed subgroup of $\S$ that  acts 1-transitively on $Y = \NN$. Suppose that   $C(G)$ is finite. Then $\Theta_\gamma \colon G/C(G) \to S_V$  is a topological embedding. \end{claim}
 \begin{proof} 
 We check the conditions (a) and (b) above.  
 
 \medskip 
 
\n  (a) Suppose that $k = [(w_0, w'_0)], n  = [(w_1, w'_1)]$. Then $\gamma(g,k)=n$ iff there are $(v_0, v'_0) \sim_C (w_0, w'_0)$ and  $(v_1, v'_1) \sim_C (w_1, w'_1)$ such that $g \cdot v_0 = v_1$ and $g \cdot v'_0 = v'_1$. Since $C(G)$ is finite this condition is open.

 \medskip 
 \n  (b) An open  neighbourhood   of  $1_{G/C}$ has the form $U/C$ wherewhere $U \sub G$ is   open   and $C \le U$.   By definition of the topology on $G$, we may assume that  $U = G_{R} C$ for some finite $R \sub Y$ (as usual $G_R$ is the pointwise stabiliser). Let $B = (R \times R) / \sim_C$. Consider $g = pC \in H$, where $p \in G$. 
 
 Suppose that  $\fa k \in B \, \gamma(g,k) = k$. This means that for each $u,v \in R$, there is a $c_{u,v} \in C$ such that $p \cdot  (u,v) = c_{u,v} \cdot  (u,v)$. Since $G$ acts faithfully and 1-transitively on $Y$, for each $c,d \in C$, and each $y \in Y$, $c\cdot  y = d \cdot y$ implies that $c=d$. Therefore given another pair $r,s \in R$, $c_{u,v} = c_{u,s} = c_{r,s}$. Let $c$ be this unique witness. Then $p \cdot u = c \cdot u$ for each $u \in R$, hence  $p \in G_R C$ and therefore $g \in U/C$.
 \end{proof}

 \begin{thm} Let   $G$ be oligomorphic.  The central quotient $G/C(G)$ is quasi-oligomorphic (i.e.\ homeomorphic to an oligomorphic group). \end{thm}
 \begin{proof} 
 It is   known (as pointed out to us by Todor Tsankov) that we may assume $G$ is 1-transitive.  To see this,  one shows that there is an open subgroup $W$ such that the left translation action $\gamma \colon G\curvearrowright G \backslash W$  of  $G$ on the left cosets of $W$ is faithful and oligomorphic, and the corresponding embedding  into a copy of $\S$ is topological. To get $W$,   let $x_1, \ldots, x_k \in \omega$ represent the 1-orbits of $G$. Let $W$ be the pointwise stabiliser of $\{x_1, \ldots, x_k\}$. If $g \in G - \{1\}$ then there is $p \in G$ and $i \le k$ such that $g \cdot (p \cdot x_i) \neq p\cdot x_i$. So $p^{-1} g p \not \in W$, and hence $g \cdot  p W \neq pW$.   So the action is faithful. The rest is routine using (a) and (b) above. 
 
 Now we can apply Claim~\ref{cl:1tr}, recalling that $C(G)$ is finite. 
 \end{proof}

 For any closed subgroup $G = \Aut(M) $ of $\S$  the higher centres  are normal, so their orbit equivalence relations are $G$-invariant. If $G$ is oligomorphic they are   definable in $M$. Hence the progression of higher centres has to stop at a finite stage for each oligomorphic group $G$.  

   \subsection{Conjugacy}
 We show that conjugacy of oligomorphic  groups is smooth.  For a closed subgroup $G$ of $ \S$, let $V_G$ denote the orbit equivalence structure. For each $n$ this structure has a $2n$-ary relation symbol, denoting the orbit relation on $n$-tuples. (One could require here that the tuples have distinct elements.)

 The following fact     holds for non-Archimedean groups in general.

 \begin{fact} \label{fact:conj}Let $G, H$ be   closed subgroups of $\S$. 
 
 \bc $G, H$ are conjugate via $\alpha$ $\LR$ $V_G \cong V_H$ via $\alpha$. \ec
 \end{fact}
 
 \begin{proof}
 \rapf  Immediate.
 
 \lapf Let $M_G$ be the canonical structure for $G$;  namely there are $k_n \le \omega$ many $n$-ary relation symbols, denoting the $n$-orbits. Let $M_H$ be the structure in the same language where the $n$-equivalence classes of $V_H$ are named so that $\alpha$ is an isomorphism $M_G \cong M_H$. Clearly $G= \Aut(M_G)$ and   $H = \Aut(M_H)$. Further, $\alpha^{-1} \Aut(M_H) \alpha = \Aut(M_G)$.
 \end{proof}

 \begin{prop} Conjugacy of oligomorphic groups is smooth. \end{prop}
\begin{proof}   The map $G \to V_G$ is   Borel because we can in a Borel way find a countable dense subgroup of $G$, which of course  has the same orbits.  Now apply Fact~\ref{fact:conj}.  For  countable structures $S$ in  a fixed  language,   mapping   $S$ to its theory $\text{Th}(S)$ is Borel.   Since the theory can be seen as  a real, for $\omega$-categorical structures,  this shows smoothness. 
\end{proof}

For corresponding structures $A,B$ with $\Aut(A) = G,\Aut(B) = H$, conjugacy of $G,H$ via $\alpha$ means that $\alpha( A) $ and $B$ have the same definable subsets. To see this, consider the case that $A$ is the canonical structure for $G$.

We note the following topological variation  of Prop~\ref{prop:Sn_normaliser}. The   notation is  introduced   above.

\begin{prop}  (i) $\Aut(V_G)$ equals  the normaliser $N_G$  of $G$  in $\S$.

(ii) If $G$ is oligomorphic then $N_G /G$ as a topological group is profinite.  \end{prop}

\begin{proof}
\n (i) 
\n $\sub:$  Let $\alpha \in \Aut(V_G)$, $\beta \in G$. Clearly $\alpha$ maps $n$-orbits to $n$-orbits, so  $\alpha(M_G)$ is a renumbering of the named $n$-orbits in $M_G$. Therefore $\beta^\alpha \in \Aut(M_G) =G$. 

\n $ \supseteq:$  Let $V$ be an $n$-orbit, and let $\alpha \in  N_G$.  If $r, s \in \NN^n$ and $r,s \in \alpha(V)$, choose $\beta \in G$ such that $\alpha \beta \alpha^{-1} (r) =s$. So $\alpha(V) $ is contained in an $n$-orbit $W$. By a similar argument, $\alpha^{-1}(W)$ is contained in an $n$-orbit. Therefore $\alpha(V) = W$ is an $n$-orbit.  Hence $\alpha \in \Aut(V_G)$.

\n (ii). Let $k_n$ be the number of $n$-orbits of $G$.  Define a continuous homomorphism $\Theta: \Aut(V_G) \to \prod_n S_{k_n}$ by $\Theta(\alpha)= f$ if  $f(n)$ is the finite permutation describing the way $\alpha$ permutes $n$-orbits (numbered in some way). Clearly $G$ equals  the kernel of $\Theta$, and $\Theta/G$ is therefore a topological embedding. Since the range is compact, its inverse is also continuous. 
A closed subgroup of a profinite group is again profinite.
\end{proof} 

The converse of (ii) may fail: $N_G/G$ can be profinite, and even trivial, without $G$ being oligomorphic. For instance, there is a countable maximal-closed subgroup of $\S$, e.g.\ $AGL_n(\mathbb Q)$, the automorphism group of the structure $\mathbb Q^n$, $n \ge 2$ with the ternary function $f(x,y,z) = x+y-z$ (Kaplan and Simon). This structure is not $\omega$-categorical. 

 We don't know at present whether every profinite group occurs that way.

\subsection{$\omega$-categorical  structures with essentially finite language} 
One says that a structure $M$ has essentially finite language if $M$ is interdefinable with a structure $\wt M$ over a language in a finite signature.   (Interdefinable means same domain and same definable relations.). We present a basic fact that  can be used to obtain an oligomorphic group that is not isomorphic to   the automorphism group of such an $\omega$-categorical structure.

\begin{lemma} The following are equivalent  for a countable structure $M$.
\bi
\item[(i)] $M$ is interdefinable with a structure $N$ in  finite language with maximum arity $k$, and quantifier elimination.

\item[(ii)] $M$ is $\omega$-categorical, and for each  $n \ge k$, each $n$-orbit of $M$ is given by its projections to $k$-orbits. 
\ei
\end{lemma}
\begin{proof} (ii) implies  (i): Let  $G = \Aut (M)$, and let $N$ be the orbit structure of $M$. Thus,  $N$ is like     $V_G$ above but has an  $n$-ary predicate for each $n$-orbit. (Note that $N$ is a Fraisse limit. $V_G$ is a reduct of $N$, and as in the  finite case above its automorphism group  is the normaliser of $G$.)

(i) implies  (ii): Clearly $N$ is $\omega$-categorical, as there are only finitely many $n$ types for each $n$. 
Each formula $\phi$  in $n\ge k$ variables is a Boolean combination of q-free formulas in $\le k$ variables. If $\phi$ describes an $n$-orbit we can assume it is a conjunction of such formulas. A formula in $k' \le k$ variables describes a finite union of $k'$-orbits. Hence the $n$-orbit is given by its projections: if two $n$ tuples have are in the same projection orbits then both or none satisfy $\phi$.

\end{proof}

      \section{Kassabov and Nies: supershort first order descriptions in  certain classes of finite groups}

 Nies and Tent~\cite{Nies.Tent:17} showed that every  finite simple group $G$ has a first-order description (in the usual language of group theory) of length $O(\log (|G|)$. This result is near optimal for the whole class of finite simple groups  because of the cyclic groups, using   a counting argument together with the prime number theorem.  We show that shorter descriptions can be obtained for  certain natural classes of finite simple groups.  This works for instance when the groups $G$ in the class have presentations of length  $O(\log(|G|)) $  and the diameter of the corresponding Cayley graph is also $O(\log)$. For instance, by this method  the alternating groups  $G$  can be described in length  $O( \log \log |G|)$.  
  
  The following definition  is    from Nies and Tent~\cite{Nies.Tent:17}. 

  \begin{definition} \label{def:compress} Let  $r \colon \NN \to \NN^+$ be  an unbounded      function.  We say  that an infinite  class $\+ C$ of finite $L$-structures   is \emph{$r$-compressible} if for each      structure $G$  in $\+ C$,  there is   a  sentence $\phi$  in $L$  such that  $|\phi| = O(r (|G|))$ and $\phi$ describes~$G$.   \end{definition}

For notational convenience,   we will use the definition \[\log m = \min \{r \colon \, 2^r \ge m\} .\]

  \begin{theorem}[\cite{Nies.Tent:17}, Thm.\ 1.2]\label{t:simple}  The class of finite simple groups is $\log$-compressible. \end{theorem}  

The  first-order formulas for generation developped   in \cite{Nies.Tent:17} will be used  in the context of presentations with  Cayley graphs of small diameter.
\begin{lemma}[\cite{Nies.Tent:17}, proof of Lemma 2.4]  \label{generation}
For each positive integers $k,v$, there exists a first-order formula $\delta_{v,k} (g;x_1,\ldots,x_k)$  of length $O(k+  v)$  in the language of groups  such that   for each group $G$,  
  ${G \models \delta_{k,v} (g;x_1,\ldots,x_k)} \LR {g = w( x_1,\ldots,x_k) }$   for some   word $w$ in $F(x_1,\ldots, x_n)$ of length at most $2^v$. 
\end{lemma}

\begin{proof}
 Let
\[\delta_{0,k}(g;x_1,\ldots,x_k) \equiv \bigvee_{1 \le j \le k}[g=x_j \  \vee g = x_j^{-1} \  \vee \ g=1].\] 
For $i>0$ let
\[
\begin{split}
\delta_{i,k}(g;x_1,\ldots,x_k) \equiv \exists u_i \exists v_i [&g = u_i v_i \ \wedge\\
&\forall w_i [(w_i = u_i \vee w_i = v_i) \rightarrow \delta_{i-1,k}(w_i;x_1,\ldots,x_k)]].
\end{split}
\]
Then  $\delta_{i,k}$ has length $O(k+i)$, and $G \models \delta_{i, k} (g;x_1,\ldots,x_k)$ if and only if $g$ can be written as a product, of length at   most $2^i$,   of $x_r$'s and their inverses.
\end{proof}

  \begin{lemma} \label{lem:short_pres} Suppose that a finite simple group $G$ has a presentation 
\bc   $ \la x_1, \ldots, x_k \mid r_1, \ldots , r_m \ra$ of length $\ell$. \ec Also suppose that the diameter of the  Cayley graph is bounded by $2^v$, that is,  each $g\in G$ has the form $ w( x_1,\ldots,x_k) $ for some  free group word of length at most $2^v$.  

 There is a sentence  $\psi$  of length $ O(v + \ell)$ describing the structure  $\la G, \ol g \ra$.  \end{lemma} 
 \begin{proof}     Let $\psi$  be  the formula
\bc   $ x_1 \neq 1 \lland \bigwedge_{1 \le i \le m}  r_i = 1 \lland    \fa y \, \delta_{k,v} (y; x_1, \ldots, x_k).$ \ec
  Replacing the $x_1,\ldots, x_k$ by new constant symbols, the models of the sentence thus obtained are the  nontrivial quotients of $G$. Since $G$ is simple, this sentence describes~$\la G, \ol g \ra$. \end{proof}

\begin{lemma} \label{lem:Ak} Suppose $S$ is a generating set of $A_k$ containing a 3-cycle, say $(1,2,3)$. Then the Cayley graph of $A_k$ with respect to $S$  has diameter $O(k^4)$. \end{lemma}
\begin{proof} $A_k$ acts $1$-transitively on the set of $3$-cycles on $\{1, \ldots , k\}$  by conjugation.  Since the number of $3$-cycles is $O(k^3)$,  each  $3$-cycle can be expressed by a word of length $O(k^3)$ in the generating set.  Any    even permutation can be written as a product of at most $k$ 3-cycles.   
\end{proof}

\begin{prop} The classes  of finite alternating/symmetric groups and of finite symmetric groups are  
$\log \circ \log$-compressible. \end{prop} 

\begin{proof} We want to describe each $A_k$, and we may assume $k > 2$. By \cite[Cor 3.23]{Guralnick:08}, $A_k$ has a presentation of length $\ell = O(\log k) = O (\log \log |A_k|)$. By construction,  one of the generators of $A_k$ in the above presentation is a 3-cycle.  Therefore the diameter of the Cayley graph is at most $O(k^4)$ by Lemma~\ref{lem:Ak}.  So we can apply   Lemma~\ref{lem:short_pres} with $v = O(\log  k)$. 
(Actually a more careful look at the generation set gives that all 3-cycles can be expressed as words of length $O( \log k)$ and the diameter of the Cayley graphs is $O(k \log k)$.  )

The case of symmetric groups is similar. One uses a transposition instead of 3-cycle. We need to take into account that the symmetric groups are not simple. Since the only nontrivial quotient has size $2$, it suffices to require in the description that the group has at least $3$ elements. 
\end{proof}

\begin{prop}  Fix a prime power $q$. The class of  groups  $\text{PSL}_n(q)$ is $\log \circ \log$-compressible. \end{prop} 
\begin{proof}
The argument is similar to the case of alternating groups. By \cite[Thm.\ A and Thm.\ 6.1]{Guralnick:08}, $\text{PSL}_n(q)$ has a presentation of length 
\bc $\ell = O(\log q + \log n) = O_q (\log \log |\text{PSL}_n(q)|)$. \ec The generating set for this presentation contains a generating set for $\text{PSL}_2(q)$ with  diameter in $O(\log q)$ and a generating set of $A_n$ with  diameter in   $O(n \log n)$. Thus, every elementary matrix in $\text{PSL}_n(q)$ can be expressed as a word of length at most $O(n\log n + \log q)$.
Finally a row reduction argument gives at any matrix in $\text{PSL}_n(q)$ is a product of at most $n^2$ elementary matrices, which implies that the diameter of the Cayley graph  is at most $O(n^3 \log n \log q)$ (this bound can be improved to $O(n^2 \log q + n^2 \log n)$ by more careful examination of all element in generating set). By Lemma~\ref{lem:short_pres} the groups $\text{PSL}_n(q)$ can be described by sentence of length $O_q(\log n) = O_q( \log \log |\text{PSL}_n(q)|	)$.
\end{proof}

\subsection{Rank 1 groups}

The result in \cite[Thm.\ 4.36]{Guralnick:08} gives a bound $O(\log q)$ for both length of presentation and diameter for groups such as $\text{SL}_2(q), \text{PSU}(3,q)$ and $\text{Sz}(q)$. Since the size of these groups is polynomial in $q$, this doesn't help to get descriptions shorter than   the ones in Theorem~\ref{t:simple}. 

If we fix the characteristic and allow descriptions in second order logic, something can be done. Recall that second order logic allows quantification over relations and functions of arbitrary arity.

\begin{prop} Fix a prime $p$. The field $\text{GF}(p^k)$ has a second-order description of length $ O (
\log k) $  \end{prop}

\begin{proof}  For $q= p^k$ the first-order   sentence  $\phi_q$ from \cite[Section 4]{Nies.Tent:17} describing $\text{GF}(q)$  says that the structure is a field of characteristic $p$ such that for all elements $x$ we have $x^{p^k}=x$ and there
is some $x$ with  $x^{p^{k-1}}\neq x$.  Now in the second order version introduce function symbols $f_1, \ldots , f_k$ such that $f_1(x) = x^p$ and $f_{i+1}(x) = f_i(f_i(x))$. Thus $f_i(x) = x^{p^{(2^i)}}$. Given these we can express that $x^{p^k} = y$ in length $O(\log k)$ using the binary expansion of $k$. 
\end{proof} 
By the biinterpretability method  described in \cite[Section 5]{Nies.Tent:17}, short descriptions of the fields imply short descriptions of the finite simple groups defined over them. For Suzuki and Ree groups, we have $p=2$ and  $p=3$, respectively.

\begin{thm} The classes  of Suzuki groups $\text{Sz}(2^{2l})$ and of small Ree groups ${}^2G_2(3^{2l+1})$ are  $\log \circ \log $ compressible in second order logic. \end{thm}

  \part{Metric spaces and descriptive set theory}

  \newcommand{\rank}{\text{rank}}

\section{Nies and Schlicht: Scott relation in Polish metric spaces}
For tuples $a, b$ in a Polish metric space, the Scott relation at level $1$ is defined as usual: for each challenge $y$ on the left  side there is response $z$ on the right  side so that the enumerated metric spaces $a, y$ and $b, z$ are isometric; similar for the sides interchanged. 
\begin{proposition} 
There is a computable Polish metric space $(X,d)$ and a computable sequence $\langle y_n^*\mid n\in\omega\rangle$ of distinct elements of $X$ such that the set 
$$C_1=\{(m,n)\in \omega\times\omega\mid y_m^* \equiv_1 y_n^* \}$$ 
is $\Pi^1_2$-complete. 
\end{proposition} 
\begin{proof} 
Note that $C_1$ is clearly a $\Pi^1_2$ set. To prove that it is $\Pi^1_2$-complete, we first fix some notation. If $A\subseteq X\times Y$ and $n\in X$, let $A_n$ denote the \emph{$n$-th slice} of $A$ and $p(A)=\{m\in X\mid \exists k (m,k)\in A\}$ its \emph{projection} to the first coordinate. For $X = \omega$, we say that $A$ is universal for a point class $\Gamma$ on $Y$ if every set in $\Gamma$ occurs as a slice. 

\begin{claim} 
There is a $\Pi^1_1$-universal set $B\subseteq \omega\times \omega^\omega$  such that $p(B)$ is $\Sigma^1_2$-complete. 
\end{claim} 
\begin{proof} 
Let $B'\subseteq \omega\times\omega\times \omega^\omega$ be universal for $\Pi^1_1$-subsets of $\omega\times\omega^\omega$. Let $\pi\colon \omega\times\omega\rightarrow\omega$ be a computable bijection and $\pi^*\colon \omega\times\omega\times\omega^\omega\rightarrow\omega\times\omega^\omega$ the induced bijection. It is easy to see that $B=\pi^*(B')$ is  universal for $\Pi^1_1$-subsets of $\omega^\omega$. Since the projection of $B'$ to $\omega\times\omega$ is universal for $\Sigma^1_2$-subsets of $\omega$, it is $\Sigma^1_2$-complete. Since $p(B' )$ has the same $1$-degree as $p(B)$ it follows that  the projection $p(B)$ is $\Sigma^1_2$-complete as well. 
\end{proof} 

We fix a $\Pi^1_1$-universal set $B$ as in the previous claim. Then $S=\omega\times\omega^\omega\setminus B$ is $\Sigma^1_1$-universal. It now follows that the equivalence relation on $\omega$ defined by $(m,n)\in E\Leftrightarrow S_m=S_n$ is $\Pi^1_2$-complete as a set, since $S_n\neq\omega^\omega\Leftrightarrow B_n\neq\emptyset \Leftrightarrow \exists x (n,x)\in B\Leftrightarrow n\in p(B)$ and $p(B)$ is $\Sigma^1_2$-complete. 

\begin{question} Show that $E$  is  $\Pi^1_2$-complete as an equivalence relation. \end{question}
\begin{claim} 
We can associate in a computable way to each $n\in\omega$ a Polish space $(Y_n,d_n)$ of diameter at most $1$ and some $y_n^*\in Y_n$ with distance set $\{d_n(y_n^*,y)\mid y\in Y_n\}=\{0\}\cup S_n$. 
\end{claim} 
\begin{proof} 
We first define an auxiliary Polish metric space $(X,d_X)$. 
Let $X=\{(0,0)\} \cup (0,1]\times\omega^\omega)$. Let   $m_{r_0,r_1}=\min\{r_0,r_1\}$ for $r_0,r_1\in\mathbb{R}$ and let  $u$ be the standard ultrametric on $\omega^\omega$. We define 
$$d_X((r_0,x_0),(r_1,x_1))=|r_0-r_1|+m_{r_0,r_1} u(x_0,x_1).$$ 
To see that $d_X$ is a metric on $X$, first note that  by the ultrametric inequality, $u(x_0,x_2)$ is at most the maximum of  $u(x_0,x_1)$ and  $u(x_1,x_2)$. So if $r_1\ge r_0$ or $r_1 \ge r_2$ then 
$$\Delta_0=(m_{r_0,r_1}u(x_0,x_1)+m_{r_1,r_2}u(x_1,x_2))-m_{r_0,r_2}u(x_0,x_2)\geq0$$ 
 We can hence assume that $r_1< r_0, r_2$ and additionally that $r_0 \leq r_2$ by symmetry between $(r_0,x_0)$ and $(r_2,x_2)$. Again  by the ultrametric inequality for $\omega^\omega$. In both cases, we have 
$$\Delta_0\geq (r_1-r_0)u(x_0,x_2)\geq r_1-r_0=-(r_0-r_1).$$
By our assumption $r_1<r_0\leq r_2$, we further have 
$$\Delta_1=(|r_0-r_1|+|r_1-r_2|)-|r_0-r_2|\geq r_0-r_1.$$ 
Hence $(X,d_X)$ satisfies the triangle inequality. 

We now define the required spaces $(Y_n,d_n)$ as subspaces of $(X,d)$. 
By identifying $\omega^\omega$ with the set of irrational numbers in $[0,1]$, we let $C_n\subseteq X $ be a closed set with $S_n=p(C_n)$. Note that we can obtain $C_n$ computably in $n$, assuming that our universal sets are constructed in the usual way. 
Now let $Y_n=\pi(C_n)$ and let $y_n^*$ be an element that is identified with $\pi(0,0)$. It is clear that $y_n^*$ has the required distance set in $(Y_n,d_n)$. 
\end{proof} 

Now let $Y=\bigcup_{n\in\omega} Y_n$ and $d_Y$ the metric on $Y$ given by the metrics $d_n$ and $d_Y(x,y)=2$ if $x\in Y_m$ and $y\in Y_n$ for some $m\neq n$. Since player II wins $G(y_m^*,y_n^*,1)\}$ if and only if $S_m=S_n$, it is now easy to see from the previous two claims that $C_1$ is $\Pi^1_2$-complete.  
\end{proof}

   
 \part{Model theory and definability}
%
%
%
 \section{Nies and Schneider: Concrete presentations, isomorphism, and descriptions}
 
 \subsection{Summary, mostly  in layman's terms.}  Mathematical structures are usually given by concrete presentations.  A computer scientist might think of a graph as a concrete object stored in a computer, for instance an adjacency list, which is a list of all the vertices and all the edges.  For another example, a set of generators together with  a set of relators on them present  a group.  
 
 What really counts is  the essence of the structure,  the structure  ``up to isomorphism": think of  the shape of the graph, or of  the abstract group.  Two concrete presentations that yield the same abstract structure  are called isomorphic. Being concrete,  the presentations can be used as input to some kind of computation. The question arises:   \begin{question}  How hard is it to tell whether two   presentations are isomorphic?   \end{question}
It is still unknown whether one can decide efficiently that two concretely presented finite  graphs are isomorphic (though Babai has recently shown that  the decision problem is in pseudopolynomial time). 
 
Related questions are  the following. Given a reference class of infinite structures,  \begin{question} which     structures are  determined  within the class by their first order theory? \end{question}  For the class of finitely generated groups, this  property  is  called \emph{quasi axiomatisable} (QA); see the last two chapters of the by now venerable survey~\cite{Nies:DescribeGroups}.   For instance, abelian groups are QA.  Even better,  \begin{question} which structures can be described within the class by  a  single  sentence in first-order logic? \end{question} 

In the same context, this  property  is  called \emph{quasi finitely axiomatisable}, or QFA \cite{Nies:DescribeGroups}.  Abelian groups are never QFA, but   other very common groups are, e.g.\ the Heisenberg group  $UT_3(\ZZ)$ or  the Baumslag-Solit\"ar group  $\ZZ[1/n] \rtimes \ZZ$.

  Even a description by  the full first-order  theory   would necessarily   only determine the essence of the structure (technically: isomorphic concrete structures are elementarily equivalent).  It is interesting to study these questions especially in the setting of topological algebra and Lie algebras, where not much has been done so far. The point is that first-order logic can only indirectly address the topology, because that is given by subsets.

 \subsection{Some more detail for mathematicians.}
 As mentioned, we have to distinguish between concrete presentations of a structure, and the abstract structure ``up to isomorphism". 
 Consider a finite presentation of a group: \bc $\la x_1, \ldots, x_n \mid r_1, \ldots r_k \ra$.  \ec This   describes the  concrete group $F(x_1, \ldots, x_n)/N$, where $N$ is the normal subgroup generated by $r_1, \ldots, r_k$. Given  two finite presentations, it is undecidable in general whether they describe isomorphic groups (Rabin).

A finite presentation of a group or Lie algebra, say,   is a description of a concrete   structure. We can also describe a structure merely  up to isomorphism. If we want to do this, we have to pick some language from mathematical logic and a corresponding satisfaction relation. First-order logic  has the additional advantage that  it doesn't look beyond the immediate structure as given by the elements and the  relations among them  (for instance,   subsets of the structure are not allowed).  This is a severe restriction, given that for instance in group theory, one frequently studies things like maximal subgroups etc. In first order logic, we can talk about particular ones if they are definable, e.g., the centre  or the centraliser of an element. But we can't quantify over the whole lot.
 
 If the structure is finite, we can look for a  short sentence, relative to the size of the structure,  describing it; e.g. Nies and Tent \cite{Nies.Tent:17} do this for (classes of)  finite groups. 
 
  If the structure is infinite,  we need the  external information given by the  reference class, because we can   not describe it by a single first-order sentence. The class needs to  at  least prescribe  the cardinality of the structure. For instance, we can describe the ordering of the rationals by a single sentence within the countable structures. A f.g.\ group is called   QFA   (for quasi-finitely axiomatisable) if,  within the class of f.g.\ groups,  it can be described  by a single sentence in the language of group theory, 

To reiterate, given a class $\+ C$ of concretely presented  structures of the same signature,
 there are two interrelated types of question \bi 
\item[(a)]  How complicated is the isomorphism relation between   structures in~$\+ C$?

\item[(b)] Which  structures can up to isomorphism  be described within $\+ C$ by their theory? 

\item[(c)] Or even by single  a first-order sentence?

\ei 

As for (a), the intuition may be  that if a concrete structure has a complicated equivalence class under the  isomorphism relation, it  is  hard to describe abstractly.   The trivial upper bound for isomorphism is $\Sigma^1_1$ (assuming the class $\+ C$ itself is arithmetical). On the other hand, elementary equivalence is easier, namely hyperarithmetical, and  in fact  $\PI 1(\ES^{\omega})$. 

Question (b) is interesting in particular if some classification of structures in $\+ C$ is known; for instance, we will below consider  the simple Lie algebras over $\mathbb C$. In this case, one would try to prove  that each structure  in the class has a description by a first-order sentence. This means that the classification can be expressed in first-order language for the  right signature, given the reference class (which may be  not  first-order axiomatisable). In this case one would hope  that all the sentences  have a common bound on the number of  quantifier. If so, this makes the isomorphism relation arithmetical: two structures are non-isomorphic iff there is a sentence at the given level of complexity that holds in one but not the other.

\subsection{Describing simple Lie algebras over $\mathbb C$ by a first-order sentence with an additional predicate}

We describe simple Lie algebras over $\mathbb C$ by a first-order sentence in the language with $+, [ , ]$ and the equivalence  relation $E$  that vectors  $x$ and $y$  span the same subspace.

There is a formula  expressing that $x_1, \ldots , x_n$ generate $L$ as a vector space:

\bc $\delta(x_1,\ldots, x_n)\equiv  \fa z  \ex y_1, \ldots , y_n   \, [ z = \sum_i y_i \land \bigwedge_i  y_i E x_i] $. \ec

Let $R_n$ denote the free associative algebra in $n$ generators $x_1, \ldots , x_n$  over $\mathbb C$; it is generated as a $\mathbb C$-vector space by all the words in the generators. The usual commutator in $R_n$ is denoted $[a,b] = ab-ba$.
With this commutator, $R_n$ becomes a Lie algebra.  The free Lie algebra in n variables over $\mathbb C$ is the Lie subalgebra of $R_n$ generated by $ x_1,,,,,,x_n$."

Each finite-dimensional Lie algebra is finitely presented, because the multiplication table on the basis elements gives a finite presentation.  So we have relators of the form $[x_i, x_j] = \sum_k c_k x_k$ where the $c_k$ are complex coefficients. (The situation is analogous to the case of finite groups.)

Using Cartan's classification, Serre proved that each semisimple Lie algebra $L$ over $\mathbb C$ is finitely presented  where coefficients   are integers in $[-3, 3]$. This can be   seen from the Cartan table; see Humphreys~\cite[Section 18.1]{Humphreys:12} for the presentation,  and also note that  the $x_i $ and $y_i$ there generate $L$.     Then there is $m$ depending on the dimension of $L$  such that only commutators of depth up to $m$  in the Lie generators $z_1, \ldots , z_k$ are needed to generate $L$ as a vector space. Thus the generators in $L$ satisfy the formula  $\alpha_m(z_1, \ldots, z_k)$ saying that each commutator of depth $m+1$ in those generators is a linear combination of the commutators of depth $m$. 

To describe a simple $L$ within the   Lie algebras over $\mathbb C$, we express that $L$ is non-trivial and that there are $z_1, \ldots z_k$ satisfying the Serre relations, the formula $\alpha_m$, and, using the formulas $\delta(x_1, \ldots, x_n)$,  that the commutators of depth $m$ in the $z_1, \ldots, z_k$  generate $L$ as a vector space. Since $L$ is simple, this sentence describes $L$. 

To obtain a description in the language of  Lie algebras we would need to define $E$ in terms of the Lie operations. First one would  show  that $E$ is invariant under automorphisms in the finite dimensional case. 


\section{Nies, Schlicht and Tent:Bi-interpretations for $\omega$-categorical structures and theories}
\newcommand{\Th}{\mathrm{Th}}
\newcommand{\bSi}{\mathbf{\Sigma}}
We discuss bi-interpretations of pairs of $\omega$-categorical theories.  We begin with structures rather than theories, because definability is easier to grasp. Definability will always mean without parameters.

\subsection{Interpretations of structures}  Suppose that $L,K $ are first-order  languages in countable signatures. Interpretations via first-order formulas of $L$-structures in $K$-structures are formally defined, for instance,   in Hodges \cite[Section 5.3]{Hodges:93}.   Informally, an $L$-structure $A$ is interpretable in  a $K$-structure $B$ if the elements of $A$ can be represented by tuples in a definable $k$-ary relation $D$ on $B$, in such a way that equality of $A$  becomes a $B$-definable equivalence relation $E$ on $D$,  and the other atomic  relations on $A$ are also definable.  
 
 We think of the interpretation of $A$ in $B$
as a \emph{decoding function}  $\Delta$. It decodes $A$ from $B$ using  first-order formulas, so that $A = \Delta(B)$ is  an $L$-structure.  Each $L$-formula $\phi$ corresponds to a $K$-formula $\psi'$ which is the saturation under $E$ of a $K$-formula $\psi$. We write $\phi = \Delta(\psi')$. 

For a structure $A$, recall that $A^{eq}$ has a sort $V= D/E$ for each definable equivalence relation $E$ on $A^n$ and definable $E$-closed $D \sub A^n$, and besides the inherited ones has definable relations $V_i$, $i< n$ between $A $ and $V$, given by 

\[ S^V_i a v  \lra   \ex \ol y  [  a = y_i \land [\ol y]_E = v ].\]
For instance if $n=1$, we have the relation $S^V _0$  that $v $ is the equivalence class of $a$.

Clearly $\Aut (A)$ acts on $A^{eq}$. The $r$-orbits on a sort $V = D/E$ have the form $[U]_{E^r}$ where $U$ is an $n\cdot r$-orbit of $A$. So if $\Aut (A)$ is  oligomorphic, there are only finitely many such $r$-orbits. 

\begin{example} {\rm Let $E$ be an equivalence relation with all classes of size $2$. Take unary predicates  $C_0, D_0, C_1, D_1$ partitioning the domain, and let $A$ be the structure where each  $E$-class   has exactly one element in $C_i$ and one in $D_i$, for some $i\leq 1$. 
Then $A$ has 4 orbits, the sort $A/E$ only has two orbits.  For the first, $U$ above can be either $C_0$ or $D_0$. } \end{example} 

Throughout we will have $\omega$-categorical structures $A,B$ with $\Aut(A) = G,\Aut(B) = H$. 
There are various equivalent views of expressing interpretation of $A$ in $B$. 

\bi  \item[(a)] $A = \Delta(B)$ for some interpretation $\Delta$, as above

\item[(b)]  A map $\alpha \colon A \to B^{eq}$ with range contained in single sort, \footnote{An alternative definition (CITE EVANS) allows the range to be a subset of finitely many sorts.} sending  relations $\ES$-definable in $A$ to  relations $\ES$-definable in $B^{eq}$. This map extends canonically to a map $\wt \alpha \colon A^{eq} \to B^{eq}$.

\item[(c)] There exists a topological homomorphism  $h \colon \Aut (B) \to \Aut (A)$ such that the range of $h$  is   oligomorphic.
\ei
 (a), (b) are merely reformulations of each other. For (a) to (c), let $h(\beta) = \beta^n \uhr D /E$ which has oligomorphic range by remarks above.
 
 For (c) to (a) see Hodges \cite[Section 7.4]{Hodges:93}. 

\subsection{Bi-interpretations of structures} \label{ss:BI theories}
 There are  several equivalent formulations. Fix structures $A$, $B$.

 (a)  $A \cong \Delta (B) , B \cong \Gamma(A)$, and some  isomorphisms $\gamma: A \cong  \Delta(\Gamma(A))$ and $\delta: B \cong  \Gamma(\Delta(B))$  are definable in $A$, in $B$, respectively.  (Thus,  $\gamma$ is described  by a formula with $1+n$ free variables, where $n$ is the product of the dimensions of the two interpretations.)  If the structures are $\omega$-categorical, we can let  $\delta$ be  the restriction of $\wt \gamma$ to the sort on which $B$ is defined.
%
 
(b)  $\alpha \colon  A \to B^{eq}, \beta \colon B \to A^{eq}$, and the maps $\gamma \colon A \to A^{eq}$ given by $\gamma = \wt \beta \circ \alpha $, and  $\delta= \wt \alpha  \circ \beta $ analogous for $B$, are definable in the respective structure. 
 
 Note that $\gamma  \colon A \to
  V$ for some sort $V$. 
  

 A bi-interpretation introduces a matching of orbits.  Suppose $\alpha $ is $d$-dimensional, and $\beta $ is $e$-dimensional.
 \begin{fact} \label{fact:match orbits} For each $n$-orbit $S$ of $A$,  $\wt \alpha (S)$ is an $n$-orbit of $B^{eq}$ (under the action of $\Aut(B)$, on the sort which contains  the range of $\alpha$.  \end{fact}
 
 \begin{proof} As usual let $G = \Aut A, H = \Aut B$. For simplicity first let $n=1$. Recall from Ahlbrandt/Ziegler \cite{Ahlbrandt.Ziegler:86}  (detail in   \href{http://wwwf.imperial.ac.uk/~dmevans/Bonn2013_DE.pdf}{ David Evans' 2013  notes}, Thm. 2.9) that the ``dual'' $\alpha^* \colon H \to G$ is a topological isomorphism, where  
 \[(\alpha^* h)  w = \alpha^{-1} (h ( \alpha(w))).\]
 So as $h$ ranges over $H$, $\alpha^*(h)$  ranges  over $G$. Let $S = G\cdot w$, then   $\alpha(S) = H \cdot \alpha(w)$ as required.
 More generally, for each $n$ we have $\alpha( G\cdot  (w_1, \ldots, w_n)) = H\cdot  (\alpha(w_1), \ldots , \alpha(w_n))$. \end{proof}
 
 \subsection{Bi-interpretations of theories}
  We can also formulate biinterpretability for complete theories $S, T$, easiest  in countable languages. Note that  theories can be seen as infinite bit sequences and hence the set of theories carries the usual Cantor space topology.  The complete theories form a closed set.  To be $\omega$-categorical is an arithmetical property of theories, because by Ryll-Nardzewski this property  is equivalent to saying that for each $n$, the Boolean algebra of formulas with at most $n$ free variables modulo $T$-equivalence is finite.

   To fix some notation, the sorts in models of $S^{eq}, T^{eq}$ have the forms $C/E, D/F$, resp,  where $C$ is an $r$-ary definable relation, $D$ is $s$-ary, and $E$, $F$ are definable equivalence relations. 
$  \gamma, \delta$ each are $(1+n)$-ary, where $n =rs$.
 Given $\Gamma, \Delta$ as above, we express that for an arbitrary  model $A$ of $S$ and $B : = \Delta(A)$, we have  $B \models T$,  and $\gamma $ evaluated in $A$ induces   an isomorphism of   $A$ and $\Gamma(\Delta(A))$ (a structure with domain  a sort of $A^{eq}$);  similarly,   $\delta$ is  an isomorphism of $B$ and $\Delta (\Gamma(B))$. This can be expressed by two  possibly   infinite lists of sentences that have to be in $S$, and in $T$, respectively.   
  
 \begin{remark} \label{rem: delta}  {\rm  Since $B$'s domain is a sort of $A$ and $B$ is $\omega$-categorical,   requiring  that $\delta$ exists is actually redundant: $\delta $ can be chosen to be ``$\Delta(\gamma)$". This means that we apply   the interpretation $\Delta$ to the definable isomorphism  $\gamma: A \cong \Gamma(\Delta(A)$, obtaining an isomorphism $\delta \colon \Delta(A) \cong \Delta(\Gamma(\Delta(A)))$, i.e. $\delta \colon B \cong \Delta(\Gamma (B))$.  Clearly $\delta$ is invariant under the $\Aut(B)$ action on $B^{eq}$. Hence $\delta$ is
 $B$-definable. }   \end{remark}
  
   \begin{fact} Suppose $S$, $T$ are bi-interpretable theories in  the notation above. For each model  $A$   of $S$,  letting   $B = \Delta(A) $,  a model  of $T$,  we have that $A, B$ are bi-interpretable as models. \end{fact} 
 
  \begin{proof}
  $\beta \colon B \to A^{eq}$ is the identity map on this concrete structure $B = \Delta(A)$. $\alpha \colon A \to B^{eq}$ we can therefore choose the same as $\gamma$, and $\wt \beta \circ \alpha = \gamma$ is  definable in $A$.   $\wt \alpha \circ \beta$ is the same as $\delta$, hence definable in $B$.   
  \end{proof}
 
 \begin{remark} \label{rem:match types} {\rm  In the case of theories rather than structures, the matching of orbits in Fact~\ref{fact:match orbits}   becomes a matching of types. Each $k$-type $\phi$ of $S$, i.e.\ an atomic formula with free variables $x_1, \ldots, x_k$,  is given by a $k$ type $\psi'$ of $T^{eq}$ in the sense that $\phi$ is $\Gamma(\psi')$, and therefore by a $k\cdot r$ type $\psi$ of $T$ whose saturation under the definable equivalence relation $F$ on $D \sub B^r$ is equivalent to $\psi'$ (note there could be various such types $\psi$). Similarly for types of $T$. }
\end{remark}

\subsection{Isomorphism of groups and bi-interpretability}

By isomorphisms of topological groups, we always mean topological isomorphisms.  Two  $\omega$-categorical structures are bi-interpretable iff their automorphism groups are isomorphic; \href{http://wwwf.imperial.ac.uk/~dmevans/Bonn2013_DE.pdf}{ David Evans' 2013  notes}, Thm. 2.9. This was originally proved by Coquand.

 \begin{theorem}  \label{thm:coq} Isomorphism of oligomorphic groups is  Borel bi-reducible with  bi-interpretability of $\omega$-categorical theories. \end{theorem}

\begin{proof} $\le_B$: From oligomorphic $G$ we can in a  Borel way determine a countable dense subgroup $\hat G$. The canonical structures for $G$ and $\hat G$ are equal. The canonical structure $M_{G}$ for $ G$ can thus be Borel determined from $G$. From $M_G$ we can Borel determine the theory $\Th(M_G)$. Then  $G \cong H$  iff $\Th(M_G)$ and $\Th(M_H)$ are bi-interpretable by  \href{http://wwwf.imperial.ac.uk/~dmevans/Bonn2013_DE.pdf}{ David Evans' 2013  notes}, Thm. 2.9.

\n $\ge_B$: From a  consistent theory $T$ in a countable signature, via the Henkin construction  we can  in a  Borel way determine a   model $M \models T$ with domain $\omega$. Let $F(T)$ be the automorphism group of such a model,  which is  a closed subgroup of $\S$. Then for $\omega$-categorical theories $S, T$, we have that $S$ is bi-interpretable with $T$ iff $F(S)$ is isomorphic to $F(T)$ by   \href{http://wwwf.imperial.ac.uk/~dmevans/Bonn2013_DE.pdf}{ David Evans' 2013  notes}, Thm. 2.9.
\end{proof}





\subsection{Bi-interpretability of $\omega$-categorical theories \\ is given by a $\SI 2$ condition}

\begin{thm} \label{thm:BISigma2} 
There  is a $\SI 2$ relation which coincides with bi-interpretability on the  $\PI 3$ set of $\omega$-categorical theories. In particular,   bi-interpretability of $\omega$-categorical theories is $\PI 3$ and hence Borel.  \end{thm} 
  \begin{proof}
\n (a)   The   initial block of existential quantifiers in the $\SI 2$ condition  states that there are   (numbers that are codes for)  sorts $V= C/E,W= D/F$ as in Subsection~\ref{ss:BI theories}    in the languages of $S, T$,   and a potential isomorphism described by a formula $ \gamma$ (which determines an isomorphism $\delta$ as in Remark~\ref{rem: delta}). 

\medskip 
\n (b) To complete the interpretations of theories, it is sufficient to provide for each $k$-type $\tau$ of $S$ (describing a $k$-orbit in any model of $S$) a $k\cdot r$ type $\rho$ of $T$, in which case  the orbit for $\tau$ goes to the orbit induced on the sort $D^k/F^k$ by $\rho$.  Similarly for $S$ and $T$ interchanged.

We now Turing compute  from the join of $S $ and $ T$ as oracles  a tree whose  maximal branching at each node is also  bounded computably in $S,T$. Any path on the tree will provide a bi-interpretation based on    the given finite information in (a). It is   $\PI  1$ in $S,T$ that there is such a branch by K\"onig's Lemma, so the whole statement is $\SI 2$ as required.

 The $k$-th level  of the tree provides matchings of  the $k$-types of $S$ with $k\cdot r$-types of $T$, and $s$-types of $T$ with $k\cdot s$-types of $S$, according to Remark~\ref{rem:match types}. From $S$, $T$ we can compute how many $k$-types there are, and how many possible matchings exists, so the branching bound is computable in the join of $S $ and $ T$.  We require that the matchings are consistent with that  the map defined by $\gamma$  describes   an isomorphism.  
 
 To see  how to do this, we  take a $k$-type of $S$. For simplicity of notation, assume that $k=2$. So we are given an atomic formula $\phi(x_1, x_2)$ for $S$. A path of length $2n$ in the tree  provides sufficient information about  $\hat A = \Gamma(\Delta (A))$. We want to show that in any bi-interpretation extending this path, $\gamma$ maps the 2-orbit defined by $\phi$ in $A$ to the 2-orbit defined by $\phi$ in $\hat A$. 
 
 Starting from $\hat A$, we have $\phi = \Gamma (\psi)$ (in the sense introduced in Subsection~\ref{ss:BI theories} above) where $\psi$ is evaluated on the sort of $B = \Delta(A)$ which contains the domain of $\hat A$. We can see  $\psi$  as a  formula with two blocks of $s$ free variables each, which is saturated under $E$.  Similarly, $\psi = \Delta(\chi )$ where  $\chi$ is saturated  under $F$  with $2s$ blocks of $r$ free variables each. Let $\bar y$ denote tuples of $n = sr $ variables, and view the  variables of $\chi$ as  two blocks of $n $ variables. The condition for $\phi$ the theory $S$ has  to satisfy  when admitting  this path of length $2n$ onto  the tree is

 \bc $\fa x_1 \fa x_2 \fa \bar y_1 \fa \bar y_2 \,  [ \gamma(x_1, \bar y_1 ) \land \gamma(x_2, \bar y_2) \to ( \phi (x_1, x_2 ) \lra \chi(\bar y_1, \bar y_2)]$. \ec
  \end{proof}

By a straightforward modification of \cite[Section 2]{Lempp.Slaman:07}, the set of $C$ countable $\omega$-categorical theories is $\Pi_3^0$-complete with respect to continuous reductions.  However, if we suitably change the topology we can make this set closed while retaining the same Borel sets.  In this way the $\omega$-categorical theories can be considered as points in   a Polish space.

\begin{cor} \label{cor:BISigma2}  Bi-interpretability on the set of $\omega$-categorical theories is Borel isomorphic to a $\bSi^0_2$-equivalence relation on a Polish space. \end{cor} 
\begin{proof} By a well known fact from descriptive set theory e.g.\ \cite[Corollary 4.2.4]{Gao:09},  there is a  finer Polish topology with the same Borel sets in which the set of $\omega$-categorical theories is closed.  Then the $\Sigma^0_2$ condition above yields a $\bSi^0_2$ description of bi-interpretability on this closed set. 
\end{proof}

  Recall that a Borel equivalence relation on a Polish space is called \emph{countable} if every equivalence class is countable.

 \begin{thm}  Isomorphism of oligomorphic groups is Borel reducible to  a countable Borel equivalence relation.
 \end{thm}
 
 \begin{proof}   A Borel equivalence relation $E$ on a Polish space $X$ is called \emph{potentially $\bSi^0_2$} if there is a finer Polish topology on $X$ with the same Borel sets in which $E$ is $\bSi^0_2$. By  Hjorth and Kechris \cite[Proposition 3.7]{Hjorth.Kechris:95}, this condition is equivalent to being Borel reducible a $\bSi^0_2$  equivalence relation on a Polish space.  
 
  By Theorem \ref{thm:coq} and Corollary~\ref{cor:BISigma2}, isomorphism of oligomorphic groups is  potentially $\bSi^0_2$. Nies, Schlicht and Tent (Oligomorphic groups are essentially countable, in preparation) showed this relation  is Borel equivalent to the isomorphism relation on a Borel invariant set of models, and hence the orbit equivalence relation of a Borel action $S_\infty$. It now suffices to apply another  result of  Hjorth and Kechris \cite[Theorem 3.8]{Hjorth.Kechris:95}: if the  orbit equivalence relation given by a  Borel action of $S_\infty$   is potentially $\bSi^0_2$, then it  is   Borel reducible to  a countable Borel equivalence relation.

%
  \end{proof}

\def\cprime{$'$} \def\cprime{$'$}

%

\end{document}